\documentclass{amsart}
\usepackage{amssymb,amsthm}
\usepackage[nobysame]{amsrefs}
\usepackage[sc]{mathpazo}
\usepackage[usesnames,svgnames]{xcolor}
\usepackage{tikz}
\usepackage[colorlinks=true,
	    urlcolor=MidnightBlue,
	    citecolor=DarkGreen]{hyperref}
\usepackage{subfigure}
\usepackage{etoolbox}
\newtheorem{theorem}{Theorem}[section]
\newtheorem{proposition}[theorem]{Proposition}
\newtheorem{lemma}[theorem]{Lemma}
\newtheorem{corollary}[theorem]{Corollary}
\theoremstyle{remark}
\newtheorem{example}[theorem]{Example}
\newtheorem{remark}[theorem]{Remark}
\definecolor{tomato}{rgb}{1,.25,.19}
\newcommand{\mtam}[2]{\ifstrequal{#2}{1}{\mathcal{T}_{#1}}{\mathcal{T}_{#1}^{(#2)}}}
\newcommand{\ps}[2]{\text{ps}_{#1}(#2)}
\newcommand{\alert}[1]{{\color{DarkGreen}\emph{#1}}}
\newcommand{\mdyck}[2]{
	\begin{tikzpicture}
		\draw(0,0) -- (0,.2) -- (#1,.2) -- (#1,.4) -- (#2,.4) -- (#2,.6) -- (1.2,.6);
		\draw[dashed](0,0) -- (1.2,.6);
	\end{tikzpicture}
}
\author{Henri M\"uhle}
\address{Fak. f\"ur Mathematik, Universit\"at Wien, Oskar-Morgenstern-Platz 1, 1090 Wien, Austria}
\email{henri.muehle@univie.ac.at}
\thanks{This work was funded by the FWF Research Grant No. Z130-N13.}
\title{The Topology of the $m$-Tamari Lattices}

\begin{document}

\begin{abstract}
	The $m$-Tamari lattices $\mathcal{T}_{n}^{(m)}$ were recently introduced by Bergeron and Pr{\'e}ville-Ratelle as
	posets on $m$-Dyck paths, and it was shown by Bousquet-M{\'e}lou, Fusy and Pr{\'e}ville-Ratelle that
	these lattices form intervals in the classical Tamari lattice $\mathcal{T}_{nm}$. It follows from a theorem by 
	Bj{\"o}rner and Wachs and a basic property of EL-shellable posets, that the $m$-Tamari lattices are EL-shellable. In this
	article, we define a new EL-labeling of the $m$-Tamari lattices completely in terms of $m$-Dyck paths. 
	With the help of this labeling, we compute the values of the M{\"o}bius function of $\mathcal{T}_{n}^{(m)}$, and we 
	characterize the intervals of $\mathcal{T}_{n}^{(m)}$ according to their topological properties.
\end{abstract}

\maketitle

\section{Introduction}
  \label{sec:introduction}
The classical Tamari lattices $\mtam{n}{1}$ as introduced in \cite{tamari62algebra}, are a well-studied 
member of the large group of Catalan objects. Bergeron and Pr{\'e}ville-Ratelle 
\cite{bergeron12higher} have recently generalized this class of lattices to $m$-Tamari lattices 
$\mtam{n}{m}$ in order to calculate the graded Frobenius characteristic of the space of higher 
diagonal harmonics. Along with this generalization, they proposed a realization of these lattices 
via $m$-Dyck paths, analogously to the realization of the classical Tamari lattices via 
classical Dyck paths. In \cite{bjorner97shellable}*{Theorem~9.2} it was shown that the classical Tamari
lattices are EL-shellable, and it is the statement of \cite{bousquet11number}*{Proposition~4} that 
$\mtam{n}{m}$ is an interval in $\mtam{nm}{1}$. Since intervals of EL-shellable posets are again 
EL-shellable, see \cite{bjorner80shellable}*{Proposition~4.2}, it is immediate that $\mtam{n}{m}$ is 
EL-shellable. The purpose of this article is to use the EL-shellability of $\mtam{n}{m}$ to compute
the M{\"o}bius function of $\mtam{n}{m}$ and to characterize the intervals of $\mtam{n}{m}$ according to
the value of their topological properties. These results are summarized in the following theorem.

\begin{theorem}\label{thm:shellability_m_tamari}
	The $m$-Tamari lattices $\mtam{n}{m}$ are EL-shellable for every positive integer $m$ and $n$.
	Moreover, the M{\"o}bius function of $\mtam{n}{m}$ takes values only in $\{-1,0,1\}$. In particular,
	every closed interval of $\mtam{n}{m}$ is homotopy equivalent to either a sphere or a point.
\end{theorem}

It is the purpose of this article to give a new, independent proof of this theorem, which works completely
within the framework of $m$-Dyck paths. It shall be remarked, that another, independent proof of
Theorem~\ref{thm:shellability_m_tamari} can be deduced from \cite{liu99left}. Liu gave a different 
EL-labeling of $\mtam{n}{1}$, which implies with \cite{mcnamara06poset}*{Theorem~3} that 
$\mtam{n}{1}$ is left-modular. Since left-modularity is a property which is inherited to intervals, see
\cite{liu99left}, the $m$-Tamari lattices are left-modular as well, and thus admit a natural EL-labeling.
Moreover, Markowsky showed in \cite{markowsky92primes}*{Theorem~22} that $\mtam{n}{1}$ is an
extremal lattice. Thus the result on the topology of $\mtam{n}{m}$ can be deduced from Thomas' results on 
trim lattices in \cite{thomas06analogue} or directly from Pallo's computation of the M{\"o}bius function of
$\mtam{n}{1}$ in \cite{pallo93algorithm}.

\smallskip

We recall the construction of $m$-Tamari lattices, as well as the
definition of EL-shellability of a poset in Section~\ref{sec:preliminaries}. In 
Section~\ref{sec:m_tamari_shellability} we give an EL-labeling of 
$\mtam{n}{m}$ by generalizing a construction from \cite{bjorner97shellable}, thus proving the first
part of Theorem~\ref{thm:shellability_m_tamari}. We conclude this article by proving the second part of
Theorem~\ref{thm:shellability_m_tamari} in Section~\ref{sec:applications}, where we characterize the intervals of 
$\mtam{n}{m}$ according to their topological properties.

\section{Preliminaries}
  \label{sec:preliminaries}
The definition of the $m$-Tamari lattices follows \cite{bergeron12higher}. For more background on EL-shellable 
posets, we refer to \cites{bjorner96shellable,bjorner97shellable}. 

\subsection{Generalized Tamari Lattices}
  \label{sec:tamari}
Let $\mathbf{a}=(a_{1},a_{2},\ldots,a_{n})$ be a sequence of integers which satisfies the conditions
\begin{align}
	\label{eq:mdyck_one}
	& a_{1}\leq a_{2}\leq\cdots\leq a_{n}\quad\text{and}\\
	\label{eq:mdyck_two}
	& a_{i}\leq m(i-1),\quad 1\leq i\leq n,
\end{align}
and denote by $\mathcal{D}_{n}^{(m)}$ the set of all these sequences.

\begin{remark}
	The sequences defined before have the following combinatorial interpretation. An \alert{$m$-Dyck path 
	of height $n$} is a path from $(0,0)$ to $(mn,n)$ in $\mathbb{N}\times\mathbb{N}$ which stays above 
	the line $x=my$, and which consists only of steps of the form $(0,1)$, so-called \alert{right-steps}, 
	or steps of the form $(1,0)$, so-called \alert{up-steps}. Given a sequence 
	$\mathbf{a}=(a_{1},a_{2},\ldots,a_{n})$ satisfying conditions \eqref{eq:mdyck_one} and 
	\eqref{eq:mdyck_two}, we can associate an $m$-Dyck path to $\mathbf{a}$ in which the $i$-th up-step 
	is followed by $a_{i+1}-a_{i}$ right-steps.	
\end{remark}
  
It is well-known (see for instance \cite{dvoretzky47problem}) that the number of $m$-Dyck paths of height $n$
is counted by the \alert{Fu{\ss}-Catalan number}
\begin{align}
	\mbox{Cat}^{(m)}(n)=\frac{1}{mn+1}\binom{(m+1)n}{n},
\end{align}
and thus it follows that $\bigl\lvert\mathcal{D}_{n}^{(m)}\bigr\rvert=\mbox{Cat}^{(m)}(n)$. 

\begin{figure}[htb]
	\centering
	\subfigure[A $4$-Dyck path of height $6$ associated to the sequence $(0,0,3,4,4,16)$.]{
	  \label{fig:good_example}
		\begin{tikzpicture}\small
			\draw[step=.4,gray!50!white,very thin](0,0) grid (9.6,2.4);
			\draw(0,0) -- (9.6,2.4);
			\draw(0,0) -- (0,.8) -- (1.2,.8) -- (1.2,1.2) -- (1.6,1.2) -- (1.6,2) -- 
			  (6.4,2) -- (6.4,2.4) -- (9.6,2.4);
			\draw(0,0) node[label=below:$0$]{};
			\draw(1.2,0) node[label=below:$3$]{};
			\draw(1.6,0) node[label=below:$4$]{};
			\draw(6.4,0) node[label=below:$16$]{};

			\draw[dotted] (0,.5) -- (7.6,2.4);
		\end{tikzpicture}}
	\subfigure[The sequence $(0,6,8,10,12,21)$ does not encode a $4$-Dyck path.]{
		\begin{tikzpicture}\small
			\draw[step=.4,gray!50!white,very thin](0,0) grid (9.6,2.4);
			\draw(0,0) -- (9.6,2.4);
			\draw(0,0) -- (0,.4) -- (2.4,.4) -- (2.4,.8) -- (3.2,.8) -- (3.2,1.2) -- 
			  (4,1.2) -- (4,1.6) -- (4.8,1.6) -- (4.8,2) -- (8.4,2) -- (8.4,2.4) -- 
			  (9.6,2.4);
			\draw(0,0) node[label=below:$0$]{};
			\draw(2.4,0) node[label=below:$6$]{};
			\draw(3.2,0) node[label=below:$8$]{};
			\draw(4,0) node[label=below:$10$]{};
			\draw(4.8,0) node[label=below:$12$]{};
			\draw(8.4,0) node[label=below:$21$]{};
			\filldraw[fill=tomato](1.6,.4) -- (2.4,.4) -- (2.4,.6) -- cycle;
			\filldraw[fill=tomato](8,2) -- (8.4,2) -- (8.4,2.1) -- cycle;
		\end{tikzpicture}}
	\caption{Some examples of paths and sequences.}
	\label{fig:explanation}
\end{figure}
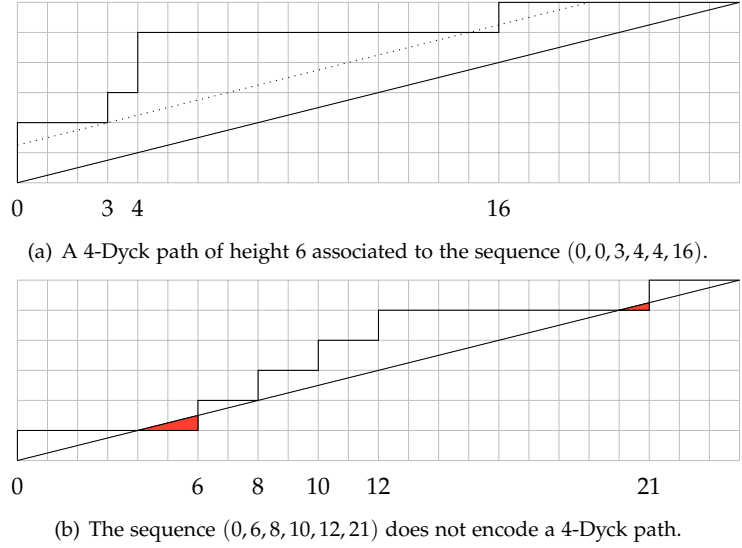

For every $i\in\{1,2,\ldots,n\}$ there exists a unique subsequence $(a_{i},a_{i+1},\ldots,a_{k})$ of
$\mathbf{a}$, called the \alert{primitive subsequence (at position $i$)} that satisfies
\begin{align*}
	& a_{j}-a_{i}<m(j-i),\quad i<j\leq k\quad\text{and}\\
	& \text{either}\quad k=n\quad\text{or}\quad a_{k+1}-a_{i}\geq m(k+1-i).
\end{align*}
The dotted line in Figure~\ref{fig:good_example} indicates that $(3,4,4)$ is the unique primitive 
subsequence at position $3$ in the $4$-Dyck path given there.

We construct a covering relation $\lessdot$ on $\mathcal{D}_{n}^{(m)}$ in the following 
way: let $\mathbf{a},\mathbf{a}'\in\mathcal{D}_{n}^{(m)}$ with $\mathbf{a}=(a_{1},a_{2},\ldots,a_{n})$, and 
let $i\in\{1,2,\ldots,n-1\}$ with $a_{i}<a_{i+1}$. Define 
\begin{multline*}
	\mathbf{a}\lessdot\mathbf{a}'\quad\text{if and only if}\\
	  \mathbf{a}' = (a_{1}\ldots,a_{i},a_{i+1}-1,a_{i+2}-1\ldots,a_{k}-1,a_{k+1},\ldots,a_{n}),
\end{multline*}
where $(a_{i+1},a_{i+2},\ldots,a_{k})$ is the unique primitive subsequence of $\mathbf{a}$ at 
position $i+1$. Denote by $\leq$ the reflexive and transitive closure of $\lessdot$. Proposition~4 in 
\cite{bousquet11number} implies that the poset $\bigl(\mathcal{D}_{n}^{(m)},\leq\bigr)$ is a 
lattice, the \alert{$m$-Tamari lattice $\mtam{n}{m}$}. 

\begin{example}
	Figure~\ref{fig:example_m32} shows the lattice of all $2$-Dyck paths of height $3$ and 
	the associated lattice of sequences as defined in the previous paragraph.

	\begin{figure}
		\centering
		\begin{tikzpicture}\small
			\def\x{.8};
			\def\y{1.2};
			\draw(3*\x,1*\y) node(p0){\mdyck{.4}{.8}};
			\draw(2*\x,2*\y) node(p1){\mdyck{.2}{.8}};
			\draw(1*\x,3*\y) node(p2){\mdyck{0}{.8}};
			\draw(3*\x,3*\y) node(p3){\mdyck{.2}{.6}};
			\draw(2*\x,4*\y) node(p4){\mdyck{0}{.6}};
			\draw(6*\x,4*\y) node(p5){\mdyck{.4}{.6}};
			\draw(3*\x,5*\y) node(p6){\mdyck{0}{.4}};
			\draw(5*\x,5*\y) node(p7){\mdyck{.2}{.4}};
			\draw(7*\x,5*\y) node(p8){\mdyck{.4}{.4}};
			\draw(4*\x,6*\y) node(p9){\mdyck{0}{.2}};
			\draw(6*\x,6*\y) node(p10){\mdyck{.2}{.2}};
			\draw(5*\x,7*\y) node(p11){\mdyck{0}{0}};
			\draw(p0) -- (p1) -- (p2) -- (p4) -- (p6) -- (p9) -- (p11);
			\draw(p0) -- (p5) -- (p7) -- (p10) -- (p11);
			\draw(p1) -- (p3) -- (p4);
			\draw(p3) -- (p7) -- (p9);
			\draw(p5) -- (p8) -- (p10);
			\draw(8*\x,4*\y) node{$\longleftrightarrow$};
			\draw(11*\x,1*\y) node(v0){024};
			\draw(10*\x,2*\y) node(v1){014};
			\draw(9*\x,3*\y) node(v2){004};
			\draw(11*\x,3*\y) node(v3){013};
			\draw(10*\x,4*\y) node(v4){003};
			\draw(14*\x,4*\y) node(v5){023};
			\draw(11*\x,5*\y) node(v6){002};
			\draw(13*\x,5*\y) node(v7){012};
			\draw(15*\x,5*\y) node(v8){022};
			\draw(12*\x,6*\y) node(v9){001};
			\draw(14*\x,6*\y) node(v10){011};
			\draw(13*\x,7*\y) node(v11){000};
			\draw(v0) -- (v1) -- (v2) -- (v4) -- (v6) -- (v9) -- (v11);
			\draw(v0) -- (v5) -- (v7) -- (v10) -- (v11);
			\draw(v1) -- (v3) -- (v4);
			\draw(v3) -- (v7) -- (v9);
			\draw(v5) -- (v8) -- (v10);
		\end{tikzpicture}
		\caption{The lattice of $2$-Dyck paths of height $3$ and the associated lattice
		  of sequences.}
		\label{fig:example_m32}
	\end{figure}
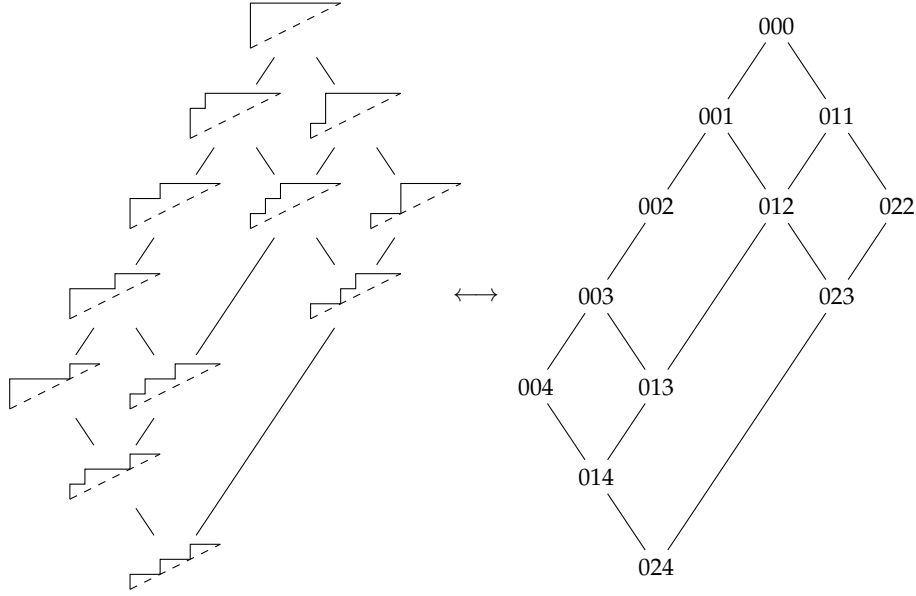
\end{example}

\subsection{EL-Shellability of Posets}
  \label{sec:shellability}
Initially, this property was introduced for graded posets, in order to create an order-theoretic 
tool to investigate shellability of posets \cite{bjorner80shellable}. 
Shellability of a poset implies several topological and order theoretical properties, for example
Cohen-Macaulayness of the associated order complex. More implications of shellability can for 
instance be found in 
\cites{bjorner80shellable, bjorner83lexicographically, bjorner96shellable, bjorner97shellable}. 
In \cite{bjorner96shellable}, Bj\"orner and Wachs generalized EL-shellability to non-graded 
posets. This is the property of interest in the present article.

Let $(P,\leq)$ be a poset. We call a poset \alert{bounded} if it has a unique \mbox{minimal} and a 
unique maximal element, denoted by $\hat{0}$ and $\hat{1}$, respectively. Denote by 
$\mathcal{E}(P)$ the set of all covering relations $p\lessdot q$ in $P$. Hence, 
$\mathcal{E}(P)$ corresponds to the set of edges in the Hasse diagram of $P$. Consider a 
non-singleton interval $[x,y]$ in $P$ and a chain $c:x=p_{0}<p_{1}<\cdots<p_{s}=y$. A chain is called 
\alert{maximal in $[x,y]$} if there are no $q\in P$ and no $i\in\{0,1,\ldots,s-1\}$ such that 
$p_{i}<q<p_{i+1}$. For some poset $\Lambda$, call a map
$\lambda:\mathcal{E}(P)\rightarrow\Lambda$ an \alert{edge-labeling of $P$} and let $\lambda(c)$ 
denote the sequence 
$\bigl(\lambda(p_{0},p_{1}),\lambda(p_{1},p_{2}),\ldots,\lambda(p_{s-1},p_{s})\bigr)$ of 
edge-labels of $c$ with respect to $\lambda$. The chain $c$ is called \alert{rising} if 
$\lambda(c)$ is a strictly increasing sequence. Moreover, $c$ is called 
\alert{lexicographically smaller} than another maximal chain $\tilde{c}$ in the same interval if 
$\lambda(c)<\lambda(\tilde{c})$ with respect to the lexicographic order on $\Lambda^{\ast}$, the 
set of words over the alphabet $\Lambda$. More precisely, the lexicographic order on 
$\Lambda^{\ast}$ is defined as $(p_{1},p_{2},\ldots,p_{s})\leq(q_{1},q_{2},\ldots,q_{t})$ if and 
only if either 
\begin{align*}
	p_{i}=q_{i}, &\quad\text{for}\;1\leq i\leq s\;\text{and}\;s\leq t\quad\text{or}\\
	p_{i}<q_{i}, &\quad\text{for the least}\; i\;\text{such that}\;p_{i}\neq q_{i}.
\end{align*}
An edge-labeling of $P$ is called \alert{EL-labeling} if for every interval 
$[x,y]$ in $P$  there exists a unique rising maximal chain in $[x,y]$, which is 
lexicographically first among all maximal chains in $[x,y]$. A bounded poset that admits an 
EL-labeling is called \alert{EL-shellable}.

\section{EL-Shellability of the $m$-Tamari Lattices}
  \label{sec:m_tamari_shellability}
Bj\"orner and Wachs have shown in \cite{bjorner97shellable}*{Section~9} that the classical 
Tamari lattices are EL-shellable, and \cite{bousquet11number}*{Proposition~4} and 
\cite{bjorner80shellable}*{Proposition~4.2} imply the same for $\mtam{n}{m}$. We will reprove this property,
using the edge-labeling
\begin{equation}\label{eq:labeling}\begin{aligned}
	\lambda:\mathcal{E}\bigl(\mtam{n}{m}\bigr) & \rightarrow\mathbb{N}\times\mathbb{N},\quad
	(\mathbf{a},\mathbf{b}) & \mapsto (j,a_{j}),
\end{aligned}\end{equation}
where $\mathbf{a}=(a_{1},a_{2},\ldots,a_{n}),\mathbf{b}=(b_{1},b_{2},\ldots,b_{n})$, as well as
$j=\min\bigl\{i\in\{1,2,\ldots,n\}\mid a_{i}\neq b_{i}\bigr\}$. 

\begin{example}
	Figure~\ref{fig:tam_32_labeled} shows the Hasse diagram of $\mtam{3}{2}$ together with the 
	edge-labeling defined in \eqref{eq:labeling}.

	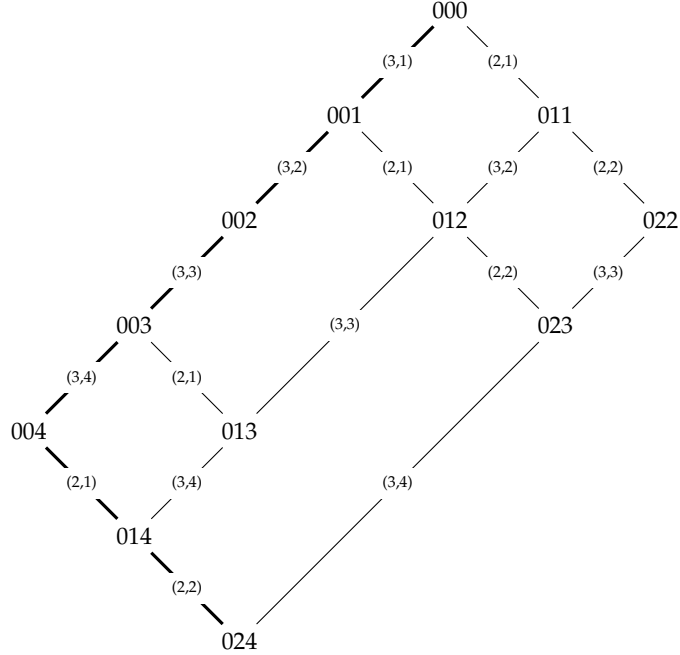
\begin{figure}[htb]
		\centering
		\begin{tikzpicture}\small
			\def\x{1.4};
			\def\y{1.4};
			\draw(3*\x,1*\y) node(p0){024};
			\draw(2*\x,2*\y) node(p1){014};
			\draw(1*\x,3*\y) node(p2){004};
			\draw(3*\x,3*\y) node(p3){013};
			\draw(2*\x,4*\y) node(p4){003};
			\draw(6*\x,4*\y) node(p5){023};
			\draw(3*\x,5*\y) node(p6){002};
			\draw(5*\x,5*\y) node(p7){012};
			\draw(7*\x,5*\y) node(p8){022};
			\draw(4*\x,6*\y) node(p9){001};
			\draw(6*\x,6*\y) node(p10){011};
			\draw(5*\x,7*\y) node(p11){000};
			\draw[very thick](p0) -- (p1) node[fill=white] at (2.5*\x,1.5*\y){\tiny(2,2)};
			\draw(p0) -- (p5) node[fill=white] at (4.5*\x,2.5*\y){\tiny(3,4)};
			\draw[very thick](p1) -- (p2) node[fill=white] at (1.5*\x,2.5*\y){\tiny(2,1)};
			\draw(p1) -- (p3) node[fill=white] at (2.5*\x,2.5*\y){\tiny(3,4)};
			\draw[very thick](p2) -- (p4) node[fill=white] at (1.5*\x,3.5*\y){\tiny(3,4)};
			\draw(p3) -- (p4) node[fill=white] at (2.5*\x,3.5*\y){\tiny(2,1)};
			\draw(p3) -- (p7) node[fill=white] at (4*\x,4*\y){\tiny(3,3)};
			\draw[very thick](p4) -- (p6) node[fill=white] at (2.5*\x,4.5*\y){\tiny(3,3)};
			\draw(p5) -- (p7) node[fill=white] at (5.5*\x,4.5*\y){\tiny(2,2)};
			\draw(p5) -- (p8) node[fill=white] at (6.5*\x,4.5*\y){\tiny(3,3)};
			\draw[very thick](p6) -- (p9) node[fill=white] at (3.5*\x,5.5*\y){\tiny(3,2)};
			\draw(p7) -- (p9) node[fill=white] at (4.5*\x,5.5*\y){\tiny(2,1)};
			\draw(p7) -- (p10) node[fill=white] at (5.5*\x,5.5*\y){\tiny(3,2)};
			\draw(p8) -- (p10) node[fill=white] at (6.5*\x,5.5*\y){\tiny(2,2)};
			\draw[very thick](p9) -- (p11) node[fill=white] at (4.5*\x,6.5*\y){\tiny(3,1)};
			\draw(p10) -- (p11) node[fill=white] at (5.5*\x,6.5*\y){\tiny(2,1)};
		\end{tikzpicture}
		\caption{The diagram of $\mtam{3}{2}$ labeled with the edge labeling as
		  defined in \eqref{eq:labeling}.}
		\label{fig:tam_32_labeled}
	\end{figure}
\end{example}

The main result of this section is stated in the following theorem.

\begin{theorem}
  \label{thm:shelling_m_tamari}
	Let $\mtam{n}{m}$ be the $m$-Tamari lattice of order $n$. The edge-labeling given in
	\eqref{eq:labeling} is an EL-labeling of $\mtam{n}{m}$ with respect to the following 
	order on $\mathbb{N}\times\mathbb{N}$
	\begin{align}\label{eq:ordering}
		(i,a_{i})\leq(j,a_{j})\quad\text{if and only if}\quad
		  i<j,\quad\text{or}\quad i=j\;\text{and}\;a_{i}\geq a_{j}.
	\end{align}
	Moreover, there is at most one falling chain in each interval of $\mtam{n}{m}$. 
\end{theorem}

\begin{proof}
	We need to show that for every interval $[\mathbf{a},\mathbf{b}]$ in $\mtam{n}{m}$ there exists
	exactly one rising chain that is lexicographically first among all maximal chains in 
	$[\mathbf{a},\mathbf{b}]$. 
	Let $\mathbf{a}=(a_{1},a_{2},\ldots,a_{n})$ and 
	$\mathbf{b}=(b_{1},b_{2},\ldots,b_{n})$ and consider the set
	$D=\{j \mid a_{j}\neq b_{j}\}=\{j_{1},j_{2},\ldots,j_{s}\}$. Let the elements of $D$ be 
	listed in increasing order, namely $j_{1}<j_{2}<\cdots<j_{s}$. Let $\mathbf{r}^{(0)}=\mathbf{a}$ 
	and construct $\mathbf{r}^{(i+1)}$ from $\mathbf{r}^{(i)}$ by decreasing the values in the 
	primitive subsequence at position $j_{k}$ of $\mathbf{r}^{(i)}$ by one, where $j_{k}$ is the 
	smallest element of $D$ such that the $j_{k}$-th entry of $\mathbf{r}^{(i)}$ is larger than
	$b_{j_{k}}$. By the minimality of $j_{k}$ it is ensured that $\mathbf{r}^{(i+1)}\leq\mathbf{b}$. 
	Since $\mtam{n}{m}$ is a finite lattice, we obtain $\mathbf{r}^{(t)}=\mathbf{b}$ after a finite 
	number, say $t$, of steps. By construction, it is clear that the chain 
	\begin{equation}\label{eq:rising_chain}
		\mathbf{a}=\mathbf{r}^{(0)}\lessdot\mathbf{r}^{(1)}\lessdot\cdots\lessdot\mathbf{r}^{(t)}
		  =\mathbf{b}
	\end{equation}
	is rising in the interval $[\mathbf{a},\mathbf{b}]$. 

	Let $\mathbf{r}^{(i)}=(r_{1}^{(i)},r_{2}^{(i)},\ldots,r_{n}^{(i)})$. It is guaranteed by 
	construction that 
	\begin{displaymath}
		\lambda\bigl(\mathbf{r}^{(i)},\mathbf{r}^{(i+1)}\bigr)=\bigl(j,r_{j}^{(i)}\bigr)
		  =\min\left\{\lambda\bigl(\mathbf{r}^{(i)},\bar{\mathbf{r}}\bigr)
		  \mid\bigl(\mathbf{r}^{(i)},\bar{\mathbf{r}}\bigr)\in\mathcal{E}\bigl(\mtam{n}{m}\bigr)\right\}. 
	\end{displaymath}
	Since the primitive subsequence of $\mathbf{r}^{(i)}$ at position $j$ is unique, any 
	other covering relation yields a label $\bigl(j',r_{j'}^{(i)}\bigr)$, with 
	$j<j'$. By following such a chain upwards, we will eventually encounter an edge 
	$(\mathbf{s},\mathbf{t})$ such that $\lambda(\mathbf{s},\mathbf{t})=(j,s_{j})$, where
	$\mathbf{s}=(s_{1},s_{2},\ldots,s_{n})$. This creates a descent in such a chain. Hence, 
	the chain in \eqref{eq:rising_chain} is the unique rising chain in $[\mathbf{a},\mathbf{b}]$ and is 
	lexicographically first.

	Now consider the set 
	$D'=\{j\mid a_{j}\neq b_{j}\;\text{and}\;a_{j}\geq a_{j-1}+m\}=\{j_{1},j_{2},\ldots,j_{t}\}$. 
	This means that $k\in D$ implies $a_{k}\neq b_{k}$ and there is no primitive 
	subsequence of $\mathbf{a}$ at position $i<k$, which contains $a_{k}$. Similarly to the previous 
	paragraph, we can see that if there exists a falling chain 
	$\mathbf{a}^{(0)}<\mathbf{a}^{(1)}<\cdots<\mathbf{a}^{(t)}$ in $[\mathbf{a},\mathbf{b}]$, it must 
	have the sequence of edge-labels
	\begin{displaymath}
		\bigl(j_{t},a_{j_{t}}\bigr),\bigl(j_{t-1},a_{j_{t-1}}\bigr),\ldots,\bigl(j_{1},a_{j_{1}}\bigr),
	\end{displaymath}
	since each of the values $a_{j_{1}},a_{j_{2}},\ldots,a_{j_{t}}$ must be decreased along any 
	maximal chain in $[\mathbf{a},\mathbf{b}]$ at least once. Hence, the given chain is the only possible 
	falling chain.
\end{proof}

\section{Applications}
  \label{sec:applications}
According to \cite{bjorner96shellable}, the EL-shellability of $\mtam{n}{m}$ has some  
consequences for the M\"obius function of $\mtam{n}{m}$ and the structure of the 
topological space associated to the intervals of $\mtam{n}{m}$. For an introduction on how to 
associate a topological space to a poset, we refer to \cite{wachs07poset}.

\begin{corollary}
	Let $[\mathbf{a},\mathbf{b}]$ be an interval of $\mtam{n}{m}$ and let $\mu$ the M\"obius function 
	of $\mtam{n}{m}$. Then, $\mu(\mathbf{a},\mathbf{b})\in\{-1,0,1\}$.
\end{corollary}
\begin{proof}
	This is a consequence of \cite{bjorner97shellable}*{Theorem~9.6},
	\cite{bousquet11number}*{Proposition~4.2} and \cite{bjorner96shellable}*{Theorem~5.7}.
\end{proof}

\begin{corollary}
	Every open interval in $\mtam{n}{m}$ has the homotopy type of either a point or a sphere.
\end{corollary}
\begin{proof}
	This is a consequence of \cite{bjorner97shellable}*{Theorem~9.6},
	\cite{bousquet11number}*{Proposition~4.2} and \cite{bjorner96shellable}*{Theorem~5.9}.
\end{proof}

We remark, that our Theorem~\ref{thm:shelling_m_tamari} also reproves these results. We can specify the 
previous results even more and characterize the spherical intervals, analogously to 
\cite{bjorner97shellable}*{Theorem~9.3}. For that purpose, consider 
$\mathbf{a}=(a_{1},a_{2},\ldots,a_{n})\in\mtam{n}{m}$ as well as $D\subseteq\{2,3,\ldots,n\}$ such that 
$a_{j}>a_{j-1}$ for all $j\in D$. Define the numbers 
\begin{multline}\label{eq:counter}
	\ps{\mathbf{a}}{j}=\bigl\lvert\{i\in D\mid i<j\;\text{and}\;
	  a_{j}-1-a_{i}<m(j-i)\;\text{and}\\
	\;a_{k}-a_{i}<m(k-i)\;\text{for all}\;i<k<j\}\bigr\rvert
\end{multline}
for all $j\in D$. Consider the sequence $\mathbf{a}_{j}^{\uparrow}\in\mtam{n}{m}$ that arises from 
$\mathbf{a}$ by decreasing the primitive subsequence of $\mathbf{a}$ at position $j\in D$ by one. 
Clearly, $(\mathbf{a},\mathbf{a}_{j}^{\uparrow})\in\mathcal{E}\bigl(\mtam{n}{m}\bigr)$ and 
$\ps{\mathbf{a}}{j}$ counts the primitive subsequences of $\mathbf{a}_{j}^{\uparrow}$ 
at some $i\in D$ with $i<j$ that contain the $j$-th entry of $\mathbf{a}_{j}^{\uparrow}$. 

\begin{theorem}\label{thm:topology_intervals}
	Let $\mathbf{a},\mathbf{b}\in\mtam{n}{m}$, with $\mathbf{a}\leq\mathbf{b}$, and
	$\mathbf{a}=(a_{1},a_{2},\ldots,a_{n})$ and $\mathbf{b}=(b_{1},b_{2},\ldots,b_{n})$. Let 
	$D=\{j\mid a_{j}\neq b_{j}\;\text{and}\;a_{j}>a_{j-1}\}$. The open interval $(\mathbf{a},\mathbf{b})$ 
	has the homotopy type of a $(\lvert D\rvert-2)$-sphere if
	\begin{align}
	\label{eq:sph_1}
	  a_{j}-1-a_{j-1}<m & \quad\mbox{implies}\quad b_{j}-b_{j-1}<m,
		   \quad\mbox{and}\\
	\label{eq:sph_2}
		&\;b_{j}=a_{j}-1-\ps{\mathbf{a}}{j},
	\end{align}
	for all $j\in D$. Otherwise, $(\mathbf{a},\mathbf{b})$ has the homotopy type of a point.
\end{theorem}
\begin{proof}
	Let $\mathbf{a}=(a_{1},a_{2},\ldots,a_{n})$ and $\mathbf{b}=(b_{1},b_{2},\ldots,b_{n})$ such 
	that $\mathbf{a}\leq\mathbf{b}$. We need to show that a falling chain exists in 
	$[\mathbf{a},\mathbf{b}]$ if and only if the conditions \eqref{eq:sph_1} and \eqref{eq:sph_2} are 
	satisfied. 

	Write the set $D$ in the form $D=\{j_{1},j_{2},\ldots,j_{s}\}$, where 
	$j_{1}<j_{2}<\cdots<j_{s}$. Let $\mathbf{r}^{(0)}=\mathbf{a}$ and construct $\mathbf{r}^{(i+1)}$ from
	$\mathbf{r}^{(i)}$ by decreasing the primitive subsequence at position $j_{s-i}$ of 
	$\mathbf{r}^{(i)}$ by one. It is clear that the chain 
	$\mathbf{r}^{(0)}<\mathbf{r}^{(1)}<\cdots<\mathbf{r}^{(s)}$ is falling. We also notice that 
	$\ps{\mathbf{r}^{(i)}}{j_{k}}=\ps{\mathbf{a}}{j_{k}}$ if $k<s-i$.

	First we show that \eqref{eq:sph_1} is equivalent to 
	$\mathbf{r}^{(k)}\leq\mathbf{b}$ for all $k\in\{1,2,\ldots,s\}$. Assume that there exists a 
	$k\in\{1,2,\ldots,s\}$ such that $a_{j_{k}}-1-a_{j_{k}-1}<m$ and 
	$b_{j_{k}}-b_{j_{k}-1}\geq m$ and $j_{k}$ is maximal in $D$ with respect to this 
	property. Consider the element $\tilde{\mathbf{a}}^{(0)}=
	\bigl(\tilde{a}_{1}^{(0)},\tilde{a}_{2}^{(0)},\ldots,\tilde{a}_{n}^{(0)}\bigr)\in
	\mtam{n}{m}$ that arises from $\mathbf{r}^{(s-k)}$ by decreasing the primitive subsequence of 
	$\mathbf{r}^{(s-k)}$ at position $j_{k}$ by one. Thus, $\tilde{\mathbf{a}}^{(0)}=\mathbf{r}^{(s-k+1)}$. 
	Construct elements $\tilde{\mathbf{a}}^{(i+1)}$ from $\tilde{\mathbf{a}}^{(i)}$ by decreasing 
	the value of the primitive subsequence of $\tilde{\mathbf{a}}^{(i)}$ at position $j_{k}-1$ 
	by one. By assumption, we know that $\tilde{a}_{j_{k}}^{(i)}$ is contained in the 
	primitive subsequence of $\tilde{\mathbf{a}}^{(i)}$ at position $j_{k}-1$. Hence, in each 
	such step, we decrease the value of $\tilde{a}_{j_{k}-1}^{(i)}$ and 
	$\tilde{a}_{j_{k}}^{(i)}$ (and possibly some subsequent entries). After a finite number, 
	say $t$, of steps, we obtain an element 
	$\tilde{\mathbf{a}}^{(t)}=(\tilde{a}_{1}^{(t)},\tilde{a}_{2}^{(t)},\ldots,\tilde{a}_{n}^{(t)})$, 
	such that $\tilde{a}_{j_{k}-1}^{(t)}=b_{j_{k}-1}$. Since $b_{j_{k}}$ is not contained 
	in the primitive subsequence of $\mathbf{b}$ at position $j_{k}-1$, we have 
	$\tilde{a}_{j_{k}}^{(t)}<b_{j_{k}}$, and thus $\mathbf{b}\leq\tilde{\mathbf{a}}^{(t)}$. Certainly,
	there is an $u\in\{0,1,\ldots,t-1\}$ such that $\tilde{a}_{j_{k}}^{(u)}=b_{j_{k}}$, and
	hence $\tilde{a}_{j_{k}-1}^{(u)}>b_{j_{k}-1}$. This implies that 
	$\tilde{a}_{j_{k}-1}^{(i)}>b_{j_{k}-1}$ for all $i\in\{0,1,\ldots,u\}$, and we can 
	conclude that $\tilde{\mathbf{a}}^{(i)}\not\leq\mathbf{b}$. Thus, 
	$\mathbf{r}^{(s-k+1)}\not\leq\mathbf{b}$. The reverse implication is trivial.

	Now we show that \eqref{eq:sph_2} is equivalent to the fact 
	that $(\mathbf{r}^{(k-1)},\mathbf{r}^{(k)})\in\mathcal{E}\bigl(\mtam{n}{m}\bigr)$ for all 
	$k\in\{1,2,\ldots,s\}$. The number $\ps{\mathbf{a}}{j_{k}}$ corresponds to 
	the number of primitive subsequences of 
	$\mathbf{r}^{(1)}=\bigl(r_{1}^{(1)},r_{2}^{(1)},\ldots,r_{n}^{(1)}\bigr)$ at position 
	$i\in D$, where $i<j_{k}$ that contain $r_{j_{k}}^{(1)}$. Hence, along the chain 
	$\mathbf{r}^{(1)}<\mathbf{r}^{(2)}<\cdots<\mathbf{r}^{(s)}$, the value of $r_{j_{k}}^{(1)}$ is 
	decreased exactly $\ps{\mathbf{a}}{j_{k}}$-times and hence 
	$b_{j_{k}}=r_{j_{k}}^{(1)}-\ps{\mathbf{a}}{j_{k}}$. By constructing 
	$\mathbf{r}^{(1)}$, we obtain $r_{j_{k}}^{(1)}=a_{j_{k}}-1$ which implies the claim.

	By combining both properties, we obtain that 
	$\mathbf{r}^{(0)}<\mathbf{r}^{(1)}<\cdots<\mathbf{r}^{(s)}$ is indeed a falling maximal chain 
	from $\mathbf{a}$ to $\mathbf{b}$.
\end{proof}

\begin{corollary}
	Let $\mathbf{a},\mathbf{b}\in\mtam{n}{m}$, with $\mathbf{a}\leq\mathbf{b}$, and
	$\mathbf{a}=(a_{1},a_{2},\ldots,a_{n})$ and $\mathbf{b}=(b_{1},b_{2},\ldots,b_{n})$. Let 
	$D=\{j\mid a_{j}\neq b_{j}\;\text{and}\;a_{j}>a_{j-1}\}$ and let $\mu$ denote
	the M\"obius function of $\mtam{n}{m}$. Then,
	\begin{displaymath}
		\mu(\mathbf{a},\mathbf{b})=\begin{cases}
					(-1)^{\lvert D\rvert}, &\;\text{if conditions}\;
					  \eqref{eq:sph_1}\;\text{and}\;
					  \eqref{eq:sph_2}\;\text{hold,}\\
					0, &\;\text{otherwise}.
		                  \end{cases}
	\end{displaymath}
\end{corollary}

\begin{corollary}
  \label{cor:mobius_zero}
	Let $\mathbf{a}\in\mtam{n}{m}$, where $\mathbf{a}=(a_{1},a_{2},\ldots,a_{n})$. Let 
	$D=\{j\mid a_{j}\neq (j-1)m\}$. Let $D_{j}=\{i\in D\mid i<j\}$, and let 
	$\mu$ denote the M\"obius function of $\mtam{n}{m}$. Then,
	\begin{displaymath}
		\mu(\hat{0},\mathbf{a})=\begin{cases}
					(-1)^{\lvert D\rvert}, &\text{if}\;a_{j}=(j-1)m-1-\lvert D_{j}\rvert\;
					  \text{for all}\;j\in D,\\
					0, &\text{otherwise}.
				\end{cases}
	\end{displaymath}
\end{corollary}
\begin{proof}
 	By construction, $\hat{0}=(0,m,2m,\ldots,(n-1)m)$. Hence, the premise 
 	in condition \eqref{eq:sph_1} corresponds to $(j-1)m-1-(j-2)m=m-1<m$ and is always 
 	satisfied. Moreover, we have 
	\begin{displaymath}
		\ps{\hat{0}}{j}=\lvert\{i\in D\mid i<j\}\rvert=\lvert D_{j}\rvert,
	\end{displaymath} 
	for all $j\in D$. If $\mathbf{a}$ satisfies \eqref{eq:sph_2}, then 
	$a_{j}=(j-1)m-1-\ps{\hat{0}}{j}$. There are two possibilities: either $j-1\in D$, or
	$j-1\notin D$.

	Consider the case that $j-1\in D$. Then, 
 	\begin{align*}
		a_{j}-a_{j-1} & =(j-1)m-1-\ps{\hat{0}}{j}-(j-2)m+1+\ps{\hat{0}}{j-1}\\
 		& =m-\ps{\hat{0}}{j}+\ps{\hat{0}}{j-1}\\
		& =m-\lvert D_{j}\rvert+\lvert D_{j-1}\rvert,
 	\end{align*}
	which yields $\lvert D_{j}\rvert-\lvert D_{j-1}\rvert>0$ for the conclusion of 
	\eqref{eq:sph_1}. Since $j-1\in D$, we know that $D_{j-1}\subsetneq D_{j}$ and the 
	conclusion of \eqref{eq:sph_1} is immediately satisfied if $a_{j}=(j-1)m-1-\lvert D_{j}\rvert$ 
	for all $j\in D$.

	Now let $j-1\notin D$. Then, $a_{j-1}=(j-2)m$, and we have
 	\begin{align*}
		a_{j}-a_{j-1} & =(j-1)m-1-\ps{\hat{0}}{j}-(j-2)m\\
 		& =m-1-\ps{\hat{0}}{j}\\
		& =m-1-\lvert D_{j}\rvert,
 	\end{align*}
	which yields $\lvert D_{j}\rvert+1>0$ for the conclusion of \eqref{eq:sph_1}. Since 
	$\lvert D_{j}\rvert\geq 0$, the conclusion of \eqref{eq:sph_1} is immediately satisfied if 
	$a_{j}=(j-1)m-1-\lvert D_{j}\rvert$ for all $j\in D$, which completes the proof.
\end{proof}

\begin{corollary}
  \label{cor:mobius_one}
	Let $\mathbf{a}\in\mtam{n}{m}$, where $\mathbf{a}=(a_{1},a_{2},\ldots,a_{n})$. Let 
	$D=\{j\mid a_{j}>a_{j-1}\}$ and let $\mu$ denote the 
	M\"obius function of $\mtam{n}{m}$. Then
	\begin{displaymath}
		\mu(\mathbf{a},\hat{1})=\begin{cases}
					(-1)^{\lvert D\rvert}, &\;\text{if}\;
					  a_{j}=\ps{\mathbf{a}}{j}+1\;
					  \text{for all}\;j\in D\\
					0, &\;\text{otherwise}.
		                  \end{cases}
	\end{displaymath}
\end{corollary}
\begin{proof}
	It follows from the definition of $\mtam{n}{m}$ that $\hat{1}=(0,0,\ldots,0)$. Hence, 
	for an interval $[\mathbf{a},\hat{1}]$ in $\mtam{n}{m}$ condition \eqref{eq:sph_1} is 
	trivially true for all $j\in D$ and condition \eqref{eq:sph_2} reduces to 
	$a_{j}=\ps{\mathbf{a}}{j}+1$ for all $j\in D$.
\end{proof}

Now we can finally conclude the proof of Theorem~\ref{thm:shellability_m_tamari}.

\begin{proof}[Proof of Theorem~\ref{thm:shellability_m_tamari}]
	This follows by definition from Theorems~\ref{thm:shelling_m_tamari} and \ref{thm:topology_intervals}.
\end{proof}

In the remainder of this section, we prove that the number of spherical intervals of $\mtam{n}{m}$
involving $\hat{0}$ or $\hat{1}$ is $2^{n-1}$ respectively.

\begin{proposition}
  \label{prop:number_spherical_intervals}
	Let $m,n\in\mathbb{N}$. Let 
	$\mathcal{S}_{n}^{(m)}(\hat{0})=\{\mathbf{a}\in\mtam{n}{m}\mid \mu(\hat{0},\mathbf{a})\neq 0\}$ and
	$\mathcal{S}_{n}^{(m)}(\hat{1})=\{\mathbf{a}\in\mtam{n}{m}\mid \mu(\mathbf{a},\hat{1})\neq 0\}$. Then,
	\begin{displaymath}
	  \label{eq:number_sphericals}
		\lvert\mathcal{S}_{n}^{(m)}(\hat{0})\rvert=2^{n-1}
		  =\lvert\mathcal{S}_{n}^{(m)}(\hat{1})\rvert.
	\end{displaymath}
\end{proposition}

Consider $\mathbf{a},\mathbf{b}\in\mtam{n}{m}$, with associated sequences
$\mathbf{a}=(a_{1},a_{2},\ldots,a_{n})$ and $\mathbf{b}=(b_{1},b_{2},\ldots,b_{n})$.
Define 
\begin{displaymath}
	D(\mathbf{a},\mathbf{b})=\bigl\{i\in \{2,\ldots,n\}\mid a_{i}\neq b_{i}\bigr\}.
\end{displaymath}
Now, for $\mathbf{a}\in\mtam{n}{m}$ and $D\subseteq\{2,3,\ldots,n\}$, and define
\begin{align}
	\text{diff}_{D}(\mathbf{a})=\{\mathbf{b}\in\mtam{n}{m}\mid D(\mathbf{a},\mathbf{b})=D\}.
\end{align}
It is immediately clear that for every $\mathbf{a}\in\mtam{n}{m}$
\begin{displaymath}
	\mtam{n}{m}=\bigcup_{D\subseteq\{2,3,\ldots,n\}}{\text{diff}_{D}(\mathbf{a})},
\end{displaymath}
and $\text{diff}_{D_{1}}(\mathbf{a})\cap\text{diff}_{D_{2}}(\mathbf{a})=\emptyset$ if and only if 
$D_{1}\neq D_{2}$.

\begin{lemma}
  \label{lem:zero_prep}
	Let $D\subseteq\{2,3,\ldots,n\}$. Then, $\text{diff}_{D}(\hat{0})\neq\emptyset$. 
\end{lemma}
\begin{proof}
	Let $D\subseteq\{2,3,\ldots,n\}$. By construction, $\hat{0}=(0,m,2m,\ldots,(n-1)m)$. Consider the 
	indicator function
	\begin{displaymath}
		\chi_{D}:\{2,3,\ldots,n\}\rightarrow\{0,1\},\quad i\mapsto\begin{cases}
			1, & \text{if}\;i\in D,\\
			0, & \text{otherwise}.\end{cases}
	\end{displaymath}
	Since $m\geq 1$, it is clear that the sequence 
	\begin{displaymath}
		\chi_{D}(\hat{0})=\bigl(0,m-\chi_{D}(2),2m-\chi_{D}(3),\ldots,(n-1)m-\chi_{D}(n)\bigr)
	\end{displaymath}
	corresponds to an $m$-Dyck path again.
\end{proof}

Let $[i,j]$ denote the interval $\{i,i+1,\ldots,j\}$ for $1\leq i\leq j\leq n$.

\begin{lemma}
	Let $i\in\{2,3,\ldots,n\}$. Then, $\text{diff}_{D}(\hat{1})\neq\emptyset$ if and only if 
	$D=[i,n]$ or $D=\emptyset$.
\end{lemma}
\begin{proof}
	By construction, $\hat{1}=(0,0,\ldots,0)$ is the sequence associated to $\hat{1}$. 
	If $D\subseteq\{2,3,\dots,n\}$ and $D\neq[i,n]$ for some $2\leq i\leq n$, then any element 
	in $\text{diff}_{D}(\hat{1})$ has to correspond to a sequence of the form 
	$(0,b_{2},b_{3},\ldots,b_{n-1},0)$. By definition it is clear that this encodes an $m$-Dyck path
	only in the case $b_{2}=b_{3}=\cdots=b_{n-1}=0$. Hence, $D=\emptyset$. 

	Now let $D=[i,n]$ for some $2\leq i\leq n$. Consider for instance the sequence 
	\begin{displaymath}
		\mathbf{a}=(\underbrace{0,0,\ldots,0}_{i-1},\underbrace{1,1,\ldots,1}_{n-i+1}).
	\end{displaymath}
	This sequence clearly encodes an $m$-Dyck path, and hence 
	$\text{diff}_{D}(\hat{1})\neq\emptyset$.
\end{proof}

\begin{proof}[Proof of Proposition~\ref{prop:number_spherical_intervals}]
	First we show that $\bigl\lvert\mathcal{S}_{n}^{(m)}(\hat{0})\bigr\rvert=2^{n-1}$. For that, let 
	$D=\{j_{1},j_{2},\ldots,j_{t}\}\subseteq\{2,3\ldots,n\}$, with $j_{1}<j_{2}<\cdots<j_{t}$. 
	Let $\mathbf{a}^{(0)}=\hat{0}$, and construct $\mathbf{a}^{(i+1)}$ from $\mathbf{a}^{(i)}$ by 
	reducing the primitive subsequence of $\mathbf{a}^{(i)}$ at position $j_{t-i}$ by one. Let 
	$\mathbf{a}^{(i)}=(a_{1}^{(i)},a_{2}^{(i)},\ldots,a_{n}^{(i)})$. 
	It is clear by construction that $a_{k}^{(i)}=a_{k}^{(0)}$ for all $k<j_{t-i+1}$. Hence, we 
	obtain a falling chain 
	$\mathbf{a}^{(0)}\lessdot\mathbf{a}^{(1)}\lessdot\cdots\lessdot\mathbf{a}^{(t)}$, 
	where $\mathbf{a}^{(t)}\in\text{diff}_{D}(\hat{0})$. (In fact, $\mathbf{a}^{(t)}$ corresponds 
	to $\chi_{D}(\hat{0})$ as defined in the proof of Lemma~\ref{lem:zero_prep}.) It follows from the 
	proof of Theorem~\ref{thm:shelling_m_tamari} that if $\mathbf{a}\in\text{diff}_{D}(\hat{0})$ and 
	there exists a falling maximal chain from $\hat{0}$ to $\mathbf{a}$, then $\mathbf{a}=\mathbf{a}^{(t)}$. 
	Hence, for every $D\subseteq\{2,3,\ldots,n\}$ we obtain exactly one 
	$\mathbf{a}\in\text{diff}_{D}(\hat{0})$, which yields 
	$\lvert\mathcal{S}_{n}^{(m)}(\hat{0})\rvert=2^{n-1}$.

	Now we show that $\bigl\lvert\mathcal{S}_{n}^{(m)}(\hat{1})\bigr\rvert=2^{n-1}$. Let $D=[i,n]$ 
	for some $2\leq i\leq n$, and let $\{j_{1},j_{2},\ldots,j_{t}\}\subseteq \{i+1,i+2,\ldots,n\}$ 
	with $i<j_{1}<j_{2}<\cdots<j_{t}$, and define $j_{0}=i$. Consider 
	$\mathbf{a}\in\text{diff}_{D}(\hat{1})$ with 
	\begin{displaymath}
		\mathbf{a}=(\underbrace{0,0,\ldots,0}_{j_{0}-1},
		  \underbrace{1,1\ldots,1}_{j_{1}-j_{0}},
		  \underbrace{2,2,\ldots,2}_{j_{2}-j_{1}},\ldots,
		  \underbrace{t+1,t+1,\ldots,t+1}_{n-j_{t}+1}).
	\end{displaymath}
	Let $\mathbf{a}^{(0)}=\mathbf{a}$ and construct $\mathbf{a}^{(i+1)}$ from $\mathbf{a}^{(i)}$ by 
	reducing the value of the primitive subsequence of $\mathbf{a}^{(i)}$ at position $j_{t-i}$ by 
	one. We obtain a chain 
	$\mathbf{a}^{(0)}\lessdot\mathbf{a}^{(1)}\lessdot\cdots\lessdot\mathbf{a}^{(t+1)}$, of 
	length $t+1$ with label sequence $\bigl((j_{t},t+1),(j_{t-1},t),\ldots,(j_{0},1)\bigr)$. 
	Moreover, $\mathbf{a}^{(t+1)}=\hat{1}$. This implies that $\mu(p,\hat{1})=(-1)^{t+1}$.
	Hence, the total number of elements $\mathbf{a}\in\text{diff}_{D}(\hat{1})$ satisfying
	$\mu(\mathbf{a},\hat{1})\neq 0$ is $2^{n-i}$. This implies 
	\begin{displaymath}
		\lvert\mathcal{S}_{n}^{(m)}(\hat{1})\rvert = 1+\sum_{i=2}^{n}{2^{n-i}}
		   = 1+\sum_{i=0}^{n-2}{2^{i}}
		   = 1+2^{n-1}-1
		   = 2^{n-1}.
	\end{displaymath}
\end{proof}

\section*{Acknowledgements}
The author is very grateful to two anonymous referees for their suggestions and helpful comments. 

\bibliography{../../literature}

\end{document}